\documentclass[a4paper,reqno]{amsart}

\usepackage{amssymb,graphicx,xcolor}
\usepackage{amssymb,bbm,enumerate}
\usepackage{esint}
\usepackage{colortbl}
\usepackage{xcolor}

\newtheorem{theorem}{Theorem}
\newtheorem{lemma}[theorem]{Lemma}
\newtheorem{con}[theorem]{Conjecture}

\newtheorem{proposition}[theorem]{Proposition}\theoremstyle{remark}

\theoremstyle{definition}
\newtheorem{definition}[theorem]{Definition}

\newcommand{\dist}{\mathrm{dist}}
\newcommand{\supp}{\mathrm{supp\,}}
\newcommand{\R}{\mathbb{R}}

\title[Falconer's distance set problem via the wave equation]
{Falconer's distance set problem\\ via the wave equation}
\date{\today}    

\author{Keith M. Rogers}
\address{Instituto de Ciencias Matem\'aticas CSIC-UAM-UC3M-UCM, Madrid, 28049, Spain.}
\email{keith.rogers@icmat.es}


\thanks{Partially supported by the grants SEV-2015-0554 and MTM2017-85934-C3-1-P (Spain).}

\begin{document}
\begin{abstract} 
Falconer proved that there are sets $E\subset \mathbb{R}^n$ of Hausdorff dimension $n/2$ whose distance sets $\{|x-y| : x,y\in E\}$ are null with respect to Lebesgue measure. This led to the conjecture that distance sets have positive Lebesgue measure as soon the Hausdorff dimension of $E$ is larger than~$n/2$.  The best results in this direction have exploited estimates that restrict the Fourier transform of measures to the $(n-1)$-dimensional sphere. Here we show that these estimates can be replaced by estimates that restrict the Fourier transform of measures to the $n$-dimensional cone. Such estimates were first considered by Wolff in their adjoint form whereby they bound the solution to the wave equation in terms of its initial data. The connection with Falconer's problem, combined with Falconer's counterexample, provides a new necessary condition for what was considered a plausible conjecture for these estimates.
\end{abstract}

\maketitle

\section{Introduction}

In 1946, Erd\"os \cite{Erd} considered the minimal number of distinct distances between~$N$ distinct points in~$\R^n$ with $n\ge 2$. By arranging the points in a lattice, he showed that the minimal number is less than  $N^{2/n}$ and conjectured that  it is not possible to substantially reduce this 
by arranging the points in a different way. The conjecture was recently confirmed by Guth and Katz in the plane~\cite{GK}.

In 1985, Falconer showed that the set of distances associated to a lattice-type structure with Hausdorff dimension $n/2$ can be null with respect to Lebesgue measure; see~\cite{F} or \cite[pp. 63]{M2}. In the converse direction he  proved that
\begin{equation*}
\mathrm{dim} (E) >n/2+1/2\quad \Rightarrow \quad |\{|x-y| : x,y\in E\}|>0.
\end{equation*}
This led to the question of what happens when the Hausdorff dimension is closer to $n/2$. The following conjecture can be found in \cite[pp. 1]{B} or \cite[pp. 58]{M2}. 

\begin{con}\label{falconer} Let $n\ge 2$ and $E\subset \mathbb{R}^n$ be a compact Borel measurable set. Then
$$
\mathrm{dim} (E) >n/2\quad \Rightarrow \quad |\{|x-y| : x,y\in E\}|>0.
$$
\end{con}

Mattila \cite{M} showed how certain decay estimates for the Fourier transform of measures can be used to attack Falconer's problem. By interpreting the decay estimates as restriction estimates on the sphere, progress was made by Sj\"olin~\cite{S} and Bourgain~\cite{B}. Further refinements were made by Wolff in two dimensions \cite{W} and Erdo\u{g}an in higher dimensions~\cite{E3}. Together their work yielded
$$
\mathrm{dim} (E) >n/2+1/3\quad \Rightarrow \quad |\{|x-y| : x,y\in E\}|>0.
$$
Beyond the application to Falconer's problem, decay estimates for the Fourier transform of measures have become objects of study in their own right; see for example~\cite{IL} and the references therein.

What is known as the \lq Knapp example' reveals that Mattila's estimates cannot confirm Conjecture~\ref{falconer} in two or three dimensions; see for example \cite{KT} or \cite{W}. In higher dimensions there is room for optimism, however there are also reasons to doubt that the best possible restriction estimates would imply the full conjecture; see for example  \cite{IR1} or~\cite{IR2}. Here we provide an alternative  approach. 

Mattila showed that the push-forward measure on the distance set is an $L^2$-function by calculating its Fourier transform. For this, the estimates that restrict the Fourier transform to $(n-1)$-dimensional spheres are useful. Instead we make use of the Riesz representation theorem to conclude that the push-forward measure is an $L^p$ function, where $p$ can be less than $2$. For this we require estimates that restrict the Fourier transform to the $n$-dimensional cone. 

The current state-of-the-art for restriction to the cone is far in advance of that for the sphere; see \cite{W2, OW}. This is \lq because' the associated multilinear estimates  are less sensitive to curvature~\cite{BCT}. The counterexamples are large in normal directions to the surface, and the normal directions to the cone are more linearly independent, or {\it transversal}, than the directions normal to the cylinder (which has the same restriction properties as the sphere).

Another potential advantage is that it is not immediately clear why this new approach could not resolve Conjecture~\ref{falconer} in lower dimensions as well. The appropriate restriction estimates to the cone are equivalent, via duality, to Strichartz estimates for the wave equation with respect to fractal measures. In fact, in order to apply them to Falconer's problem we will need the estimates in this adjoint form. Strichartz estimates with respect to fractal measures were first noted as objects of interest by Wolff \cite{W1,W2}. His estimates were improved by Erdo\u{g}an~\cite{E1}, Oberlin~\cite{O} (see also \cite{OO})  and Cho--Ham--Lee \cite{CHL}, who also calculated a number of necessary conditions, none of which prevent a confirmation of Conjecture~\ref{falconer} via this approach. Indeed, Falconer's example provides a new necessary condition for the Strichartz estimates.

A connection between Falconer's problem and the wave equation was recently observed by Liu \cite{Liu}. He showed that the Erdo\u{g}an--Wolff estimates imply a stronger conclusion for the pinned distance sets; that
$$
\dim\big(\big\{ x\in\mathbb{R}^n\,:\, \big|\{|x-y| : y\in E\}\big|=0\big\}\big) \le \tfrac{3}{2}n+1-2\dim(E)
$$
whenever $\dim(E)\in [n/2,(n+1)/2]$. This followed from an $L^2$-estimate which, by Plancherel's identity, is equivalent to a Strichartz estimate for the wave equation (this should be compared with the results of Section~\ref{five}). Here we use pointwise approximations and the Riesz representation theorem to show that results of this type could also follow from $L^p$-estimates.

\section{Definitions and preliminaries}

 We begin by recalling some standard notations regarding the Fourier transform and its inverse, defined as usual by
$$
\widehat{f}(\xi):=\int_{\R^n} e^{-2\pi ix\cdot\xi } f(x)\, dx,\qquad g^\vee(x):=\int_{\R^n} e^{2\pi ix\cdot\xi } g(\xi)\, d\xi. 
$$
Differentiating under the integral, we have $-\Delta f=\big( 4\pi^2|\cdot|^2 \widehat{f}\, \big)^\vee$, and so more generally we write $$m(-\Delta)f:=\big( m(4\pi^2|\cdot|^2) \widehat{f}\, \big)^\vee$$ and define  the fractional Sobolev or Bessel potential spaces by $$H^s(\R^n):=(1-\Delta)^{-s/2}L^2(\R^n)\quad \text{with norm}\quad \|f\|_{H^s}:=\|(1-\Delta)^{s/2}f\|_{L^2}.$$

We will consider real and even solutions $u$ to the wave equation $\partial_{tt}u=\Delta u$ with initial data in these spaces. Specifically we consider the solutions 
\begin{equation*}
u(x,t):=\cos\big(t\sqrt{-\Delta}-\tfrac{n-1}{4}\pi\big)u_0(x).\end{equation*} We will prove $L^p$-estimates for these solutions with respect to fractal measures. That is, positive, compactly supported, Borel measures~$\mu$ that satisfy the growth condition $$
[\mu]_{\alpha} := \sup_{\substack{x \in \R^n \\ r > 0}}
\frac{\mu(B(x,r))}{r^{\alpha}} < \infty.
$$
Here and throughout, $B(x,r)$ denotes the open ball, centred at $x$ and of radius~$r$. 

In the Appendix A we will see that, for any fixed time $t\in \mathbb{R}$,\footnote{We write $A \lesssim B$ if $ A\leq C B$ for some constant $C>0$ which may depend on the diameter of the support of $\mu$ and may change from line to line.}
\begin{equation}\label{ftt}
\Big(\int|u(x,t)|^p\,d\mu(x)\Big)^{1/p} \lesssim [\mu]^{1/p}_{\alpha}\|u_0\|_{H^s(\R^n)},\quad s>\frac{n-\alpha}{2},
\end{equation}
where $1\le p\le 2$ and $0<\alpha<n$.
In the following definition we consider how much smoother the solution becomes after integrating locally in time. These estimates are similar to the {\it local smoothing} estimates considered in  \cite[Corollary 3.3]{MSS}, however they can also be viewed as Strichartz estimates with respect to fractal measures.

\begin{definition}\label{demo} Let $\gamma_n(\alpha)$ denote the supremum of  $\gamma\ge 0$ such that, for some $p>1$, 
$$
\Big(\int_0^1\int |u(x,t)|^p  d\mu(x)dt\Big)^{1/p} \lesssim [\mu]^{1/p}_{\alpha}\|u_0\|_{H^{\frac{n-\alpha}{2}-\gamma}(\R^n)}
$$
whenever $u_0$ and $\mu$ are real and even.
\end{definition}

\section{Falconer's problem via Strichartz estimates}

Combining the following theorem with Falconer's example, we obtain upper bounds for the smoothing in the range  $(n-1)/2< \alpha< n/2$. Indeed we see that $\inf_{\alpha<n/2}\gamma_n(\alpha)\le 1/2$ which contradicts what was considered a plausible conjecture for the Strichartz estimates with $p<2$; see \cite[pp. 62]{CHL} or the following section.
On the other hand, by simply integrating the fixed-time estimate, we see that $\gamma_n(\alpha)$ is at least non-negative from which we recover Falconer's positive result.  Therefore, improvements to this would follow by quantifying the smoothing. In particular Conjecture~\ref{falconer} would follow by proving that there is at least half a derivative of smoothing when $\alpha=n/2$. In the fifth section,  will prove lower bounds for the smoothing as soon as $\alpha<n$, obtaining the half derivative for all $\alpha\le (n-1)/2$.

\begin{theorem}\label{wells}
Let $\frac{n-1}{2}<\alpha<\frac{n+1}{2}$ and suppose that  $\gamma_n(\alpha)\ge \frac{n+1}{2}-\alpha$. Then 
$$
\mathrm{dim} (E) >\alpha\quad \Rightarrow \quad |\{|x-y| : x,y\in E\}|>0.
$$
\end{theorem}

\begin{proof} If $\dim(E)>\alpha'>\alpha$, then Frostman's lemma tells us that $E$ supports a positive measure that satisfies
$
[\mu]_{\alpha'}<\infty$. As in Mattila's  approach, we consider something similar to the push-forward of $\mu\times\mu$ under the distance map $(x,y)\to |x-y|$ defined by
$$
\nu(F)= \mu\times\mu \{\,(x,y)\,:\, |x-y|\in F\,\}.
$$
We hope to prove that  $\nu$ is absolutely continuous with respect to Lesbesgue measure,  allowing us to conclude that its support, that is the distance set, has positive Lebesgue measure (if $d\nu(t)$ were equal to $h(t)dt$ with $h$ zero almost everywhere, then $\nu$ would be identically zero, contradicting that $\nu([0,\infty))=\|\mu\|^2>0$).  In order to achieve this we will use the Riesz representation rather than calculating the Fourier transform of $\nu$ as in \cite{M}.

As there are fewer compactly supported measures satisfying $[\mu]_{\alpha'}<\infty$ than those that satisfy $[\mu]_{\alpha}<\infty$, by definition the function given by $\alpha\mapsto \gamma(\alpha)+\alpha/2$ is increasing, so that $\alpha\mapsto\gamma(\alpha)+\alpha$ is strictly increasing. Therefore, after relabelling $\alpha'$ by  $\alpha$, it will suffice to prove that the measure $\nu_\lambda$, with $\lambda \in \mathbb{R}$, defined by
$$
\int g(t)\, d\nu_\lambda(t)= \int g(|x-y|)\,|x-y|^\lambda d\mu(x)d\mu(y),
$$ is a function whenever  $[\mu]_{\alpha}<\infty$ and
\begin{equation}\label{reft}
\gamma_n(\alpha)> \frac{n+1}{2}-\alpha.
\end{equation}
For this we will prove that
\begin{equation}\label{it}
\Big|\int g(t)\, d\nu_\lambda(t)\Big| \le C\big([\mu]_{\alpha},\|\mu\|\big)\,\|g\|_{p'}, 
\end{equation}
for some $p>1$, whenever  $[\mu]_{\alpha}<\infty$ and \eqref{reft} holds. Then  the measure $\nu_\lambda$ 
gives rise to a bounded linear functional
which, by the Riesz representation \cite[pp. 284]{R}, can be rewritten
$$
\int g(t)\, d\nu_\lambda(t)=\int g(t)\, h(t)\, dt
$$
where $\|h\|_p\le C([\mu]_{\alpha},\|\mu\|)$. As $\nu_\lambda$ is not identically zero, and $h$ is supported on the distance set, this implies that the distance set has positive Lebesgue measure.

Noting that both  the Hausdorff dimension and distance set of $E$ are unchanged under translations, we can position $E$ away from the origin. The implicit constants throughout depend on the diameter of the support of~$E$, an unimportant parameter. From now on we suppose that this diameter is bounded by one, and we position the set so that the distance between $E$ and the origin is greater than one. Then we can extend the measure $\mu$ in an even manner, so that $\mu(-x)=\mu(x)$. This does not change the value of $[\mu]_\alpha$ and only doubles the size of $\|\mu\|$. Now
 \begin{equation*}
\int g(t)\, d\nu_\lambda(t) = \frac{1}{2}\int\!\!\int \mathbf{1}_{[0,1]}(|x-y|)g(|x-y|)\,|x-y|^\lambda d\mu(x)d\mu(y), 
\end{equation*}
and so we can work with the extended, even measures from now on. 
  
Letting $\phi$ be a smooth and real function 
supported in $[-1,1]$ and equal to one on $[-1/2,1/2]$,  we decompose $\mu=\sum_{k\ge 0} \mu_k$ into smooth functions defined by
\begin{equation}\label{fourier}\widehat{\mu_0}(\xi)=\phi(|\xi|)\widehat{\mu}(\xi),\qquad\widehat{\mu_k}(\xi)=\big(\phi(2^{-k}|\xi|)-\phi(2^{-k+1}|\xi|)\big)\widehat{\mu}(\xi),\end{equation}
where  $k\ge 1$ and 
$$
\widehat{\mu}(\xi)=\int_{\R^n} \cos(2\pi x\cdot\xi)\, d\mu(x).
$$
Letting $\psi$ denote the Fourier transform of $\phi(|\cdot|)-\phi(2|\cdot|)$, we have the convolution representation 
$$\mu_k =2^{nk}\psi(2^k\cdot)\ast \mu,\quad k\ge 1.$$ 
Note that the functions $\mu_k$ are real-valued and even. It is also easy to calculate that they are bounded \begin{equation}\label{bound}\|\mu_k\|_\infty\lesssim 2^{(n-\alpha)k}[\mu]_{\alpha},\end{equation} using the rapid decay of $\psi(2^k\cdot)$ away from $B(0,2^{-k})$; see  the Appendix A. 
 
Letting $\sigma_t$ denote the surface measure on the sphere of radius $t$, by Fubini's theorem and H\"older's inequality, we can write
\begin{align*}
\Big|\int g(t)\, d\nu_\lambda(t)\Big|&\le \sum_{k\ge 0}\Big|\int\!\!\int_0^1 g(t)\, t^\lambda\sigma_t\ast \mu_k(x)\,dt\, d\mu(x)\Big|\\
&\le \sum_{k\ge 0}\|\mu\|^{1/p'}\|g\|_{p'}\Big(\int\!\!\int_0^1 \big| t^\lambda \sigma_t\ast \mu_k(x)\big|^pdtd\mu(x)\Big)^{1/p}.
\end{align*}
 Thus, to complete the proof, it will suffice to bound the integrals of the right-hand side by constants that decay sufficiently fast  in $k$ so that we can sum, yielding~\eqref{it}.

 First we note that  by \eqref{bound} we easily have
 $|t^\lambda \sigma_t\ast \mu_k(x)|\lesssim t^{n-1+\lambda }2^{(n-\alpha)k}[\mu]_{\alpha},
$
so that
$$
\Big(\int\!\!\int_0^{2^{-nk/\lambda}} |t^\lambda\sigma_t\ast \mu_k(x)|^p dt d\mu(x)\Big)^{1/p}\lesssim 2^{-\alpha k}[\mu]_{\alpha}\|\mu\|^{1/p}.
$$
Thus, we can sum a geometric series in $k$ to bound this part of the integral.  This constitutes the whole integral when $k=0$ and so we consider $k\ge 1$ from now on.

By standard properties of the Fourier transform of the surface measure on the sphere, see for example \cite[pp. 338 \& 347]{stein}, we can bound $\big|\sigma_t\ast \mu_k\big|$ by a constant multiple of
\begin{align*}\label{link}
&\ (2^{-k}t)^\frac{n-1}{2}\sum_{j\ge0} (2\pi2^kt)^{-j}\Big(\big|e^{it\sqrt{-\Delta}} \mu_k(x)\big|+\big|e^{-it\sqrt{-\Delta}} \mu_k(x)\big|\Big)\\
\lesssim\ &\ 
2^{-\frac{n-1}{2}k}|\cos\big(t\sqrt{-\Delta}-\tfrac{n-1}{4}\pi\big)\mu_k|+2^{-\frac{n+1}{2}k}2^{\frac{n}{\lambda}k}\Big(\big|e^{it\sqrt{-\Delta}} \mu_k\big|+\big|e^{-it\sqrt{-\Delta}} \mu_k\big|\Big)
\end{align*}
in the range $2^{-nk/\lambda}\le t\le 1$. Integrating this, we are led to consider estimates for the wave equation where time is real, but space is fractal. 

First we consider the second term which can be considered to be a remainder term. By the forthcoming Lemma~\ref{ft}, we have that
$$
\Big(\int|e^{\pm it\sqrt{-\Delta}} f(x)|^p\,d\mu(x)\Big)^{1/p} \lesssim [\mu]^{1/p}_{\alpha}\|f\|_{H^s(\R^n)},
$$
for all $t\in \mathbb{R}$ and all $s>(n-\alpha)/2$.
Interpolating between $\|\mu_k\|_1\le \|\mu\|$ and \eqref{bound}, we have
\begin{equation}\label{top}
\|\mu_k\|_{2}\le \|\mu_k\|_1^{1/2}\|\mu_k\|_\infty^{1/2}\lesssim 2^{\frac{n-\alpha}{2}k}\|\mu\|^{1/2}[\mu]_{\alpha}^{1/2},
\end{equation}
and so by plugging this into the fixed-time estimate, we obtain
\begin{equation*}
\Big(\int_0^1 \int |e^{\pm it\sqrt{-\Delta}} \mu_k(x)|^p d\mu(x)dt\Big)^{1/p} \lesssim 2^{sk}\|\mu\|^{1/2}[\mu]_{\alpha}^{1/p+1/2},
\end{equation*}
for all $s>n-\alpha$. Multiplying both sides by $2^{-\frac{n+1}{2}k}2^{\frac{n}{\lambda}k}$, we can take $\lambda$ sufficiently large, given that $\alpha>(n-1)/2$, so that the resulting power of $2^{k}$ is negative and we can sum a geometric series.

To deal with the main term, it would suffice to prove that
\begin{equation}\label{tr}
\Big(\int_0^1 \int \big|\cos\big(t\sqrt{-\Delta}-\tfrac{n-1}{4}\pi\big)\mu_k(x)\big|^p d\mu(x)dt\Big)^{1/p} \lesssim 2^{sk}[\mu]_{\alpha}^{1/p}\|\mu_k\|_{2},
\end{equation}
for some 
$
s<(\alpha-1)/2$. This is because, after multiplying both sides by $2^{-\frac{n-1}{2}k}$ and plugging in \eqref{top}, we would be able to sum a geometric series again, completing the proof of \eqref{it}. Now our assumption \eqref{reft} gives us this estimate, completing the proof.
\end{proof}

\section{Necessary conditions for Strichartz estimates}

In much of the previous literature, the Strichartz estimates have been considered without imposing product structure on the measure; 
\begin{equation}\label{itt}
\Big(\int\!\!\int |e^{it\sqrt{-\Delta}} f(x)|^p  d\nu(x,t)\Big)^{1/p} \lesssim [\nu]^{1/p}_{\alpha}\|f\|_{H^{s}(\R^n)}.
\end{equation}
Here the estimate is supposed to hold uniformly for all $f\in H^{s}(\R^n)$ and all compactly supported and positive Borel measures $\nu$ that satisfy $$
[\nu]_{\alpha} := \sup_{\substack{x \in \R^{n+1} \\ r > 0}}
\frac{\nu(B(x,r))}{r^{\alpha}} < \infty.
$$
Adapting a two-dimensional calculation of Erdo\u{g}an~\cite{E1}, Cho, Ham and Lee \cite[Proposition 1.5]{CHL} proved that $s\ge s(\alpha,p,n)$ is necessary for \eqref{itt} to hold, where
$$
s(\alpha,p,n):=\left \{
\begin{array}{rccccccl}
  \!\!\!\!&\max\Big\{ \frac{n}{2}-\frac{\alpha}{p},\frac{n+1}{4}\Big\}   &\mathrm{when}\quad& 0<\!\!\!&\!\! \alpha\!\!&\!\!\!\le\!\!\!&\!\!\!\!\!\!\!\!\!\!\!\! 1&\\ [0.8ex] 
  \!\!\!\!&\max\Big\{ \frac{n}{2}-\frac{\alpha}{p},\frac{n+1}{4}-\frac{\alpha-1}{2p},  \frac{n+2}{4}-\frac{\alpha}{4}\Big\}   &\mathrm{when}\quad& 1<\!\!\!&\!\! \alpha\!\!&\!\!\!\le\!\!\!&\!\!\!\!\!\!\!\!\!\!\!\! n&\\ [0.8ex] 
  \!\!\!\!&\max\Big\{ \frac{n}{2}-\frac{\alpha}{p},\frac{n+1}{4}-\frac{2\alpha-(n+1)}{2p},  \frac{n+1}{2}-\frac{\alpha}{2}\Big\}   &\mathrm{when}\quad& n<\!\!\!&\!\! \alpha\!\!&\!\!\!\le\!\!\!&\!\! n+1.&\\ [0.8ex] 
\end{array}\right.
$$ 
Noting that $s(\alpha+1,p,n)< (n-2)/4$ for certain $p<2$ and  $\alpha<n/2$,  we proved a new necessary condition in the previous section, however we will now go further. 

In the opposite direction, with $n=2$ or $3$, it is known that \eqref{itt} holds if
\begin{equation}\label{try}
s>\left \{
\begin{array}{rccccccl}
  s(\alpha,2,n)   &\mathrm{when}\quad& 1\le\!\!\!&\!\! p\!\!&\!\!\!\le\!\!\!&\!\!\!\!\! 2&\\ [0.8ex]
 s(\alpha,p,n)   &\mathrm{when}\quad& 2\le\!\!\!&\!\! p\!\!&\!\!\!\le\!\!\!&\!\!\! \infty,& \\ [0.8ex] 
\end{array}\right.
\end{equation}
where $0<\alpha\le n+1$.
The two-dimensional result is due to Erdo\u{g}an~\cite[Corollary~1]{E1} (see the Appendix B for the range $0<\alpha<1$) and the three-dimensional result is due to Cho, Ham and Lee~\cite[Theorem 1.2]{CHL}.
In fact we will see that \eqref{try} is almost necessary as well. Indeed the Strichartz estimate \eqref{itt} can fail in any dimension if $s<s(\alpha,2,n)$ with $1\le p\le 2$. 

An inspection of the definition of $s(\alpha,p,n)$ reveals that to see this it will suffice to prove that $s\ge (n-\alpha)/2$ is necessary when $0<\alpha<n$. For a contradiction we suppose otherwise and take $\alpha'$ such that  $s<(n-\alpha')/2<(n-\alpha)/2$. Then by Frostman's lemma, an $\alpha'$-dimensional compact set $E\subset \mathbb{R}^d$ supports a measure $\mu$ such that $[\mu]_{\alpha}<\infty$. Letting $d\nu(x,t)=d\mu(x)d\delta_0(t)$, where $\delta_0$ denotes the Dirac delta, we have that  $[\nu]_{\alpha}=[\mu]_{\alpha}<\infty$. On the other hand, letting $\psi$ be a cut-off function, equal to one on $E$,  the function $f$, defined by
$$
f(x):=\psi(x)\dist(E,x)^{-\lambda},
$$
 is a member of  $L^2(\R^n)$ whenever $\lambda<(n-\alpha')/2$; see \cite[Lemma 3.6]{Mo}. Then, assuming \eqref{itt}, we would have
\begin{align*}
\Big(\int |(1-\Delta)^{-s/2}f(x)|^p  d\mu(x)\Big)^{1/p}
&=\Big(\int\!\!\int |e^{it\sqrt{-\Delta}}(1-\Delta)^{-s/2} f(x)|^p  d\nu(x,t)\Big)^{1/p}\\
&\le [\nu]^{1/p}_{\alpha}\|(1-\Delta)^{-s/2}f\|_{H^{s}(\R^n)}\\
&\le [\mu]^{1/p}_{\alpha}\|f\|_{L^2(\R^n)}<\infty.
\end{align*}
Now if we take $\lambda$ greater than $s$, as we may, then $(1-\Delta)^{-s/2}f$ is still singular on~$E$, the support of $\mu$; see~\cite{Z}. Thus, the left-hand side would be infinite, an apparent contradiction, and so we conclude that $s\ge (n-\alpha)/2$ is necessary for \eqref{itt} to hold.

Returning to Falconer's problem, we are only interested in measures of the form $d\nu(x,t)=d\mu(x)\mathbf{1}_{[0,1]}(t)dt$, and so we need not be entirely discouraged by the previous necessary conditions. We can however deduce necessary conditions for our more specific problem via the following lemma in the spirit of Sobolev embedding.

\begin{lemma}\label{ittt} Suppose that
\begin{equation}\label{strichartz}
\Big(\int_0^1\int |e^{it\sqrt{-\Delta}} f(x)|^p  d\mu(x)dt\Big)^{1/p} \lesssim [\mu]^{1/p}_{\alpha}\|f\|_{H^{s}(\R^n)}
\end{equation}
whenever $f\in H^{s}(\R^n)$ and $[\mu]_\alpha<\infty$.  Then, for all $s'>s+1/p$, 
\begin{equation}\label{maximal}
\Big(\int \sup_{t\in[0,1]}|e^{it\sqrt{-\Delta}} f(x)|^p  d\mu(x)\Big)^{1/p} \lesssim [\mu]^{1/p}_{\alpha}\|f\|_{H^{s'}(\R^n)}.
\end{equation}
\end{lemma}

\begin{proof}
By the Fundamental Theorem
of Calculus,
$$
|F(t)|^{p}-|F(\tau)|^{p}\le {p}\int_{0}^1 |F(y)|^{{p}-1}
|F'(y)|\,dy.
$$
for all $t,\tau\in [0,1]$. Integrating in $\tau$, and by H\"older's
inequality,
$$
\sup_{t\in [0,1]}|F(t)|^p-\|F\|^p_{L^{p}[0,1]}\le 
p\|F\|_{L^{p}[0,1]}^{p-1}\|F'\|_{L^{p}[0,1]}.
$$
By definition $\partial_te^{it\sqrt{-\Delta}}f=i\sqrt{-\Delta}
e^{it\sqrt{-\Delta}}f$, so that taking $F(t)=e^{it\sqrt{-\Delta}} f$ and then integrating with respect to $d\mu(x)$, 
\begin{multline*}\label{nl}
\big\|\sup_{t\in[0,1]}|e^{it\sqrt{-\Delta}} f|\,\big\|^p_{L^p(d\mu)}-\|e^{it\sqrt{-\Delta}} f\|^p_{L^p(d\mu\times[0,1])}\\
\le
p\|e^{it\sqrt{-\Delta}} f\|_{L_x^p(d\mu\times[0,1])}^{p-1}\|\sqrt{-\Delta}
e^{it\sqrt{-\Delta}} f\|_{L^p(d\mu\times[0,1])},
\end{multline*}
again by H\"older's inequality. The proof is completed by dyadically decomposing the Fourier transform of $f$, applying the Strichartz inequality \eqref{strichartz} to each piece, and then summing a geometric series.
\end{proof}

Now the maximal estimates are stronger than \eqref{itt} with $d\nu(x,t)=d\mu(x)d\delta_0(t)$ and so we can recover a version of the necessary condition just proved, losing $1/p$ derivatives.
On the other hand, 
$s\ge \frac{n+2-\alpha}{4}$
is also necessary for the maximal estimate. Although the example of \cite[pp. 15-16]{CHL} does not involve a measure of the form $d\nu(x,t)=d\mu(x)d\delta_0(t)$, by projecting their measure orthogonally to the time axis and normalising appropriately, it is easy to adjust the argument to see that the same condition is necessary for the maximal estimate~\eqref{maximal}. The point is that the gain of insisting on a product structure in the measure is offset  by taking the supremum. By Lemma~\ref{ittt},  the corresponding necessary condition for our more regular Strichartz estimate~\eqref{strichartz} is then
\begin{equation}\label{yes}
s\ge \max \Big\{ \frac{n-\alpha}{2}, \frac{n+2-\alpha}{4}\Big\}-\frac{1}{p}.
\end{equation}
As before, this can be strengthened for larger $p$ via the Knapp example however here  we are only interested in $p\le 2$. 
Now Falconer's example and Theorem~\ref{wells} tell us that $s\ge \frac{n-2}{4}$ when $\alpha< n/2$ and this is larger than the right-hand side of \eqref{yes} in the range  $$\max\Big\{\frac{n+2}{2}-\frac{2}{p}, 4-\frac{4}{p}\Big\}<\alpha<\frac{n}{2}.$$ So this is a new necessary condition when $n=2$ and $p<4/3$, when $n=3$ and $p<8/5$, or when $n\ge 4$ and $p<2$. This reflects that solutions to the wave equation can be large, at many different times,  on small sets with many equidistant points.

\section{Taking $p=2$}\label{five}

Falconer's example suggests that more than half a derivative of smoothing may not be possible, at least near the critical dimension $\alpha=n/2$.  Thus, consideration of small integration exponents may not smooth the solution further and so we concentrate now on the case $p=2$ and hope to prove estimates of the form
\begin{equation}\label{con2}
\Big(\int_0^1\int |e^{it\sqrt{-\Delta}}f(x)|^2  d\mu(x)dt\Big)^{1/2} \lesssim [\mu]^{1/2}_{\alpha}\|f\|_{H^{s}(\R^n)}.
\end{equation}
We will see that this estimate holds, for all $s>\frac{n-\alpha-1}{2}$, in the range $0<\alpha\le \frac{n-1}{2}$. For a proof in the larger range $0<\alpha<n-1$ with $\mu(x)=|x|^{\alpha-n}$, see \cite{OR}.  

By Lemma~\ref{ittt} of the previous section, if \eqref{con2} held for all $s>\frac{n-\alpha-1}{2}$, then we would be able to  strengthen the fixed-time estimate~\eqref{ftt} to the maximal estimate
$$
\Big(\int \sup_{0<t<1}|e^{it\sqrt{-\Delta}}f(x)|^2  d\mu(x)\Big)^{1/2} \lesssim [\mu]^{1/2}_{\alpha}\|f\|_{H^{s}(\R^n)},\quad s>\frac{n-\alpha}{2},
$$
thus  improving the estimates of \cite{BBCR, LR}. This in turn yields bounds for the divergence sets;
$$
\dim\Big\{ x\in \mathbb{R}^n\, :\, \lim_{t\to 0} e^{it\sqrt{-\Delta}}f(x)\not= f(x) \Big\}\le \alpha,\ \ \forall\ f\in H^s(\R^n),
$$
whenever $s>(n-\alpha)/2$; see \cite[Appendix B]{BR}. The fixed-time estimate is almost sharp with respect the regularity and so we cannot expect more than half a derivative of smoothing in \eqref{con2}. 

A similar result to Theorem~\ref{wells} was proven by Mattila \cite{M, M2} for his  exponents $\beta_n(\alpha)$ defined to be the supremum of $\beta\ge 0$ for which\begin{equation}\label{dz}
\|\widehat{\mu}(R\,\cdot\,)\|^2_{L^2(\mathbb{S}^{n-1})}\lesssim
\|\mu\|[\mu]_{\alpha}R^{-\beta}\end{equation} whenever 
$R> 1$ and $[\mu]_{\alpha}<\infty$. He proved that if $\beta_n(\alpha)\ge n-\alpha$, then 
$$
\mathrm{dim} (E) >\alpha\quad \Rightarrow \quad |\{|x-y| : x,y\in E\}|>0.
$$
In the following proposition, we see that a resolution of Falconer's conjecture would pass through the new approach at least as easily.

\begin{proposition}\label{well} Let $0<\alpha\le n$. Then
$$\gamma_n(\alpha)\ge \frac{\beta_n(\alpha)+1-\alpha}{2}.$$
\end{proposition}

\begin{proof}
The main ingredient is polar coordinates and Plancherel's identity. Indeed, taking $p=2$ and noting that
\begin{equation*}\label{told}
\big|\cos\big(t\sqrt{-\Delta}-\tfrac{n-1}{4}\pi\big)\big|^2\le \big|e^{i(t\sqrt{-\Delta}-\frac{n-1}{4}\pi)}\big|^2,
\end{equation*}
by Plancherel's identity in the time variable (integrating over the whole of $\R$ at this point), we obtain
\begin{align*}
\int\!\!\int \big|\cos\big(t\sqrt{-\Delta}-\tfrac{n-1}{4}\pi\big) f(x)\big|^2  d\mu(x)dt\lesssim \int_0^\infty\int \big|\big( \widehat{f}(r\cdot) \sigma\big)^\vee(rx)\big|^2d\mu(x)r^{2(n-1)}dr.
\end{align*}
Now, writing  \eqref{dz} in a dual form (see \cite[pp. 610]{BBCR} for the details), we have 
\begin{equation*}
\|(g\sigma)^\vee (R \, \cdot \,) \|^2_{L^{2}(d \mu)}
\lesssim 
R^{-\beta} [\mu]_{\alpha} \| g\|^2_{L^{2}(\mathbb{S}^{n-1})}.
\end{equation*}
for all $\beta<\beta_n(\alpha)$.
Thus we see that
\begin{align*}
\int\!\!\int \big|\cos\big(t\sqrt{-\Delta}-\tfrac{n-1}{4}\pi\big) f(x)\big|^2  d\mu(x)dt&\lesssim  [\mu]_{\alpha}\int_0^\infty \| \widehat{f}(r\cdot)\|^2_{L^{2}(\mathbb{S}^{n-1})}r^{2(n-1)}r^{-\beta}dr\\
&\lesssim [\mu]_{\alpha}\|f\|^2_{H^{\frac{n-1-\beta}{2}}},
\end{align*}
completing the proof.
\end{proof}

Combining with the best known lower bounds for $\beta_n$,  we obtain the following lower bounds for the local smoothing:
$$
\gamma_n(\alpha)\ge\left\lbrace
\begin{array}{lll}
\frac{1}{2},& \alpha\in(0,\frac{n-1}{2}],& \\
&& \text{(Mattila~\cite{M})}\\
\frac{1}{2}-\frac{2\alpha-n+1}{4},& \alpha\in[\frac{n-1}{2},\frac{n}{2}],& \\
&&\\
\frac{1}{4}-\frac{2\alpha-n}{8},& \alpha\in[\frac{n}{2},\frac{n}{2}+1],&  \text{(Erdo\u{g}an/Wolff~\cite{E3,W})}\\
&&\\
\frac{(n-\alpha)^2}{2(n-1)(2n-\alpha-1)}, &\alpha\in[\frac{n}{2},n],&\text{(Luc\`a--Rogers~\cite{LR}).}\end{array}
\right.
$$
From Mattila's bound we deduce that \eqref{con2} is true, for all $s>\frac{n-\alpha-1}{2}$, in the range $0<\alpha\le \frac{n-1}{2}$. Combining the bounds of Wolff and Erdo\u{g}an with Proposition \ref{wells}, we \lq recover' the best known result for Falconer's problem. The bound of Luc\`a and the author reveals that there is smoothing as soon as $\alpha<n$.


\section{Falconer's problem via null-form estimates}

Solutions $u$ of the wave equation, $\big(\partial_{tt}-\Delta\big)u =0$,  satisfy
\begin{equation}\label{null}
\big(\partial_{tt}-\Delta\big) |u|^2=2\big(|\partial_t u|^2-|\nabla u|^2\big).
\end{equation}
Null-form estimates attempt to take advantage of the cancelation on the right-hand side; see for example \cite{FK,  LRV, LV}. In other words, or from a different point of view, they lend themselves to attack via bilinear estimates \cite{T2, Te, W2}. 
 
\begin{definition}\label{def2} Writing $u(x,t):= e^{it\sqrt{-\Delta}} u_0(x)$, we let $\gamma_n^\star(\alpha)$ denote the supremum of the $\gamma$ such that
$$
\int_0^1\int_{\R^n} \Big( |\partial_t u(x,t)|^2-|\nabla u(x,t)|^2\Big) (-\Delta)^{-1}\mu(x)\,dxdt \lesssim [\mu]_{\alpha}\|u_0\|^2_{H^{\frac{n-\alpha}{2}-\gamma}(\R^n)}.
$$
\end{definition}

If it were not for the positive weight $(-\Delta)^{-1}\mu$, which is bounded when $\alpha>n-2$, the left-hand side of the inequality would be identically zero, by Plancherel's identity. Note also that we only need to bound the integral rather than the modulus of the integral. Finally note that the inequality can be rewritten in terms of the Riesz transform whereupon it becomes clear that the integrand can be negative.


\begin{proposition} Let $0<\alpha\le n-2$. Then 
$$
\gamma_n(\alpha)\ge \min\{\gamma_n^\star(\alpha),1/2\}. 
$$
\end{proposition}

\begin{proof}
Using Parseval's identity and \eqref{null}, solutions of the wave equation satisfy
\begin{align*}
\int|u(x,t)|^2  d\mu(x) &= -\int\Delta |u(x,t)|^2  (-\Delta)^{-1}\!\mu(x)\,dx\\
&= \int \Big(2|\partial_t u(x,t)|^2-2|\nabla u(x,t)|^2-\partial_{tt}|u(x,t)|^2\Big) (-\Delta)^{-1}\!\mu(x)\,dx.
\end{align*}
After integrating this identity in time, it will remains to  adequately bound the third term on the right-hand side.  To this end, we apply Fubini's theorem and the Fundamental Theorem of Calculus, so that
\begin{align*}
& \ \Big|\int_0^1\int \partial_{tt}|u(x,t)|^2 (-\Delta)^{-1}\mu(x)\,dxdt\Big|
\\
 =&\  \Big|\int \Big(\partial_{t}|u(x,1)|^2-\partial_{t}|u(x,0)|^2\Big) (-\Delta)^{-1}\mu(x)\, dx\Big|\\
\lesssim &\ [\mu]_{\alpha}\Big|\int (-\Delta)^{\frac{n-\alpha'}{2}-1}\Big(\partial_{t}|u(x,1)|^2-\partial_{t}|u(x,0)|^2\Big)\,dx\Big|,
\end{align*}
for all $\alpha'<\alpha$. In the final inequality we use Parseval's identity and the fact that $(-\Delta)^{-\frac{n-\alpha'}{2}}\mu\lesssim [\mu]_{\alpha}$.  This is clear when $|x|>3$, as $(-\Delta)^{-\frac{n-\alpha'}{2}}\mu(x)\le \|\mu\|$. Otherwise, we have that
\begin{align*}
(-\Delta)^{-\frac{n-\alpha'}{2}}\mu(x)
\lesssim\ \int_{\R^n} \frac{d\mu(y)}{|x-y|^{\alpha'}}
\le  \sum_{j\ge-3} [\mu]_{\alpha}  2^{-j\alpha}2^{j\alpha'}
\lesssim  [\mu]_{\alpha},
\end{align*}
by a suitable dyadic decomposition.

Now given that 
$$
\partial_t|u(x,t)|^2=i\sqrt{-\Delta}u(x,t)\overline{u(x,t)}-iu(x,t)\overline{\sqrt{-\Delta}u(x,t)}, 
$$
in order to discard this term it would suffice to prove the fixed-time estimate
\begin{equation}\label{ret}
\int\Big|(-\Delta)^{\frac{n-\alpha'}{2}-1}\Big(\sqrt{-\Delta}u(x,t)\overline{u(x,t)}\,\Big)\,\Big|\,dx \lesssim \|u_0\|^2_{H^{\frac{n-\alpha'}{2}-\gamma}(\R^n)}
\end{equation}
for all $\gamma<1/2$.
 For this we applying the fractional Leibniz inequality 
$$
\|(-\Delta)^{s}(gh)\|_{1}\le \|(-\Delta)^sg\|_2\|h\|_2+\|g\|_2\|(-\Delta)^sh\|_2,\quad s\ge 0;
$$
see for example \cite{KPV}. We then partition the Fourier transform of $u_0$ into dyadic pieces as in \eqref{fourier} and sum a geometric series, using the fact that $\|u(x,t)\|_2=\|u_0\|_2$, to complete the proof.
\end{proof}

\section*{Appendix A: A generalised trace theorem}

Using Plancherel's identity, the fixed-time estimates are a consequence of the following lemma. Taking $p=2$ and $\mu$ to be surface measure, so that $\alpha=n-1$,  we see that this is a generalisation of the trace theorem. On the other hand, taking $\mu$ to be the Dirac delta, so that $\alpha=0$, the norm on the left-hand side resembles the supremum norm, in which case this can also be compared to  Sobolev embedding.
\begin{lemma}\label{ft} Let $1\le p\le 2$ and $s>\frac{n-\alpha}{2}$. Then
$$
\|f\|_{L^p(d\mu)}\lesssim \|\mu\|^{1/p-1/2}[\mu]^{1/2}_\alpha\|f\|_{H^{s}(\R^n)}.
$$
\end{lemma}

\begin{proof} By H\"older's inequality, it will suffice to prove the inequality with $p=2$. Note that the Fourier transform of $|f|^2=f\overline{f}$ is supported in $B(0,2R)$ whenever the Fourier transform of $f$ is supported in $B(0,R)$. Thus, by Parseval's identity, 
\begin{equation}\label{dryer}
\|f\|_{L^2(d\mu)}=\Big(\int |f(x)|^2 \mu_R(x) dx\Big)^{1/2},
\end{equation}
whenever $\widehat{f}\subset \supp B(0,R)$, where
$
\mu_R=R^n\psi(R\,\cdot)\ast \mu.
$
Here $\psi$ is the Fourier transform of $\phi(|\cdot|)$ with $\phi$ smooth, equal to one on $[-2,2]$ and supported in $[-4,4]$. Now
$$
\psi(R\,\cdot\,) \lesssim  2^{-nk}\sum_{k\ge 0} \mathbf{1}_{B(0,R^{-1}2^{k})},
$$
and 
$$
\mathbf{1}_{B(0,R^{-1}2^{k})}\ast \mu(x)=\mu\big(B(x,R^{-1}2^{k})\big)\le [\mu]_\alpha R^{-\alpha}2^{\alpha k}
$$
so that  we can sum to obtain  
$
\mu_R\lesssim [\mu]_\alpha R^{n-\alpha}.
$
Plugging this into \eqref{dryer}, and summing another geometric series completes the proof.
\end{proof}

\section*{Appendix B: Strichartz estimates from decay estimates}

Letting $\Gamma$ denote the surface measure on the cone $\{ (\xi,|\xi|)\in \mathbb{R}^{n+1}\, : 1\le |\xi|< 2\}$ and $0<\alpha\le(n-1)/2$,  it is straightforward to adapt the argument of Mattila \cite{M} to prove that
\begin{equation}\label{dzz}
\|\widehat{\nu}(R\,\cdot\,)\|^2_{L^2(d\Gamma)}\lesssim
\|\nu\|[\nu]_{\alpha}R^{-\alpha}\end{equation} whenever 
$R> 1$ and $[\nu]_{\alpha}<\infty$; see for example  \cite{S2}. By a duality argument, Strichartz estimates with respect to fractal measures imply decay estimates of this form. In the opposite direction, it is relatively easy to prove that decay estimates imply Strichartz estimates with $p=1$, however the estimates also hold in the range $1\le p\le 2$ with the same regularity. 

\begin{lemma}\label{jack-up} Let $1\le p\le 2$, $0<\alpha\le\frac{n-1}{2}$ and $s>\frac{n-\alpha}{2}$. Then
$$
\Big(\int\!\!\int |e^{it\sqrt{-\Delta}} f(x)|^p  d\nu(x,t)\Big)^{1/p} \lesssim \|\nu\|^{1/p-1/2}[\nu]^{1/2}_\alpha\|f\|_{H^{s}(\R^n)}.
$$
\end{lemma}

\begin{proof}
By writing $g=g_1-g_2 + i(g_3-g_4)$, with $g_j$ positive, and then applying the triangle inequality, \eqref{dzz} implies that
$$
\|\widehat{g\nu}(R\,\cdot\,)\|^2_{L^2(d\Gamma)}\lesssim
\|\nu\|[\nu]_{\alpha}R^{-\alpha}\|g\|^2_{L^\infty(d\nu)}.
$$
The dual estimate for the adjoint operator can be written as
\begin{equation*}
\|(h\Gamma_{\!R})^\vee\|^2_{L^1(d\nu)}\lesssim
\|\nu\|[\nu]_{\alpha}R^{n-\alpha}\|h\|^2_{L^2(d\Gamma_{\!R})},
\end{equation*}
where $\Gamma_{\!R}$ is the surface measure on $\{ (\xi,|\xi|)\in \mathbb{R}^{n+1} : R\le |\xi| < 2R\}.$ Setting $\widehat{f}(\xi)=h(\xi,|\xi|)$, this can be rewritten as
\begin{equation}\label{ll1}
\int\!\!\int |e^{it\sqrt{-\Delta}} f(x)| d\nu(x,t)\lesssim
\|\nu\|^{1/2}[\nu]^{1/2}_{\alpha}R^{\frac{n-\alpha}{2}}\|f\|_{2}
\end{equation}
whenever $\supp \widehat{f}\subset B(0,2R)\backslash B(0,R)$. 

It remains to prove the estimate with $p=2$; the other estimates follow by H\"older's inequality. Consider the sets 
$
E:=\{\, (x,t)\, : \, |e^{it\sqrt{-\Delta}} f(x)|\ge \lambda\,\}
$
and replace $\nu$ by the normalised measures $\nu(E)^{-1}\nu$ in \eqref{ll1}. Then, by Chebyschev's inequality, we obtain the weak (2,2) inequality
\begin{equation}\label{trot}
 \nu(\{\, (x,t)\in\R^{n+1}\, : \, |e^{it\sqrt{-\Delta}} f(x)|\ge \lambda\,\}) \lesssim \lambda^{-2}R^{n-\alpha}[\nu]_\alpha\|f\|^2_{2},
\end{equation}
whenever $\supp \widehat{f}\subset B(0,2R)\backslash B(0,R)$.
Now by an application of H\"older's inequality it is easy to see that
\begin{equation}\label{retro}
|e^{it\sqrt{-\Delta}} f(x)|\lesssim R^{n/2} \|f\|_{2}
\end{equation}
whenever $\supp \widehat{f}\subset B(0,2R)$, allowing us to limit the range of integration in the layer-cake representation;
$$
\int\!\!\int |e^{it\sqrt{-\Delta}} f(x)|^q d\nu(x,t)= q\int_0^{cR^{n/2}\|f\|_{2}} \lambda^{q-1}  \nu(\{\, (x,t)\, : \, |e^{it\sqrt{-\Delta}} f(x)|\ge \lambda\,\})\, d\lambda.
$$
Plugging \eqref{trot} into the representation and integrating, we obtain a bound  for all $q>2$. Taking $q$ close to $2$, and plugging the estimate into the inequality
$$
\|F\|_{L^2(d\nu)}\lesssim \|F\|^{1-\frac{q'}{2}}_{L^1(d\nu)}\|F\|^{\frac{q'}{2}}_{L^q(d\nu)}
$$
along with  \eqref{ll1} yields 
\begin{equation*}
\Big(\int\!\!\int |e^{it\sqrt{-\Delta}} f(x)|^2 d\nu(x,t)\Big)^{1/2}\lesssim R^{s} [\nu]^{1/2}_\alpha \|f\|_{L^2(\R^n)},\qquad s>\frac{n-\alpha}{2},
\end{equation*}
whenever $\supp \widehat{f} \subset B(0,2R)\backslash B(0,R)$.
Here we have used that $\|\nu\|\lesssim [\nu]_\alpha$ when the measures are compactly supported. After a dyadic decomposition,  the triangle inequality and  summing a geometric series, using \eqref{retro} to deal with the low frequencies, the result follows.
\end{proof}

\end{document}